\documentclass[final]{siamltex}
\usepackage{amsfonts,graphicx}

\begin{document}
\newtheorem{remark}[theorem]{Remark}
\newtheorem{example}[theorem]{Example}
\renewcommand{\thefootnote}{\fnsymbol{footnote}}

\footnotetext[2]{Corresponding author. Faculty of Mathematical Sciences, University of Guilan, Rasht, Iran\\  Email: khojasteh@guilan.ac.ir, salkuyeh@gmail.com}
\footnotetext[3]{Faculty of Mathematical Sciences, University of Guilan, Rasht, Iran\\  Email: vedalat.math@gmail.com}
\footnotetext[4]{Faculty of Mathematical Sciences, University of Guilan, Rasht, Iran\\  Email: hezari\_h@yahoo.com}

\renewcommand{\thefootnote}{\arabic{footnote}}

\title{A New Preconditioner for the GeneRank Problem}

\author{Davod Khojasteh Salkuyeh \footnotemark[2]\
\and Vahid Edalatpour\footnotemark[3]\
\and Davod Hezari\footnotemark[4]}

\maketitle

\begin{abstract}
Identifying key genes involved in a particular disease is a very important problem which is considered in biomedical research. GeneRank model is based on the PageRank algorithm that preserves many of its mathematical properties. The model brings together gene expression information with a network structure and ranks genes based on the results of microarray experiments combined with gene expression information, for example from gene annotations (GO). In the present study, we present a new preconditioned conjugate gradient algorithm to solve GeneRank problem and study its properties. Some numerical experiments are given  to show the effectiveness of the suggested preconditioner.
\end{abstract}



\pagestyle{myheadings}
\thispagestyle{plain}
\markboth{D. K. Salkuyeh, V. Edalatpour and D. Hezari}{A New Preconditioner for the GeneRank Problem}

\section{Introduction}
\label{SEC1}

Identifying genes involved in a particular disease is regarded as a great challenge in post-genome medical research. Such identification can provide us with a better understanding of the disease. Furthermore, it is often considered as the first step in finding treatments for it. However, the genetic bases of many multifactorial diseases are still uncertain, and modern technologies usually report hundreds or thousands of genes related to a disease of interest.
In this context is where gene-disease prioritization methods are of use.

The act of finding the most potentially successful genes among a variety of listed genes has been defined as the gene prioritization problem. Considering rapid growth in biological data sources containing gene-related information such as, for instance, sequence information, microarray expression data, functional annotation data, protein-protein interaction data, and the biological and medical literature, we can see much interest in recent years in developing bioinformatics approaches that can analyze this data and help with the identification of important genes. The common aim in the present study is to prioritize the genes in a way that those related to the disease under study possibly appear at the top of ranking.

In the last decade, several methods have been proposed for ranking or prioritizing genes by relevance to a disease. Some of these methods have been collected at Gene Prioritization Portal: \verb"http://homes.esat.kuleuven.be/~bioiuser/gpp/index.php". These methods fall into two broad classes. The first class of the methods mostly uses microarray expression data; these methods focus on identifying genes that are differentially expressed in a disease, and use simple statistical measures such as the $t$-statistic or related classification methods in machine learning to rank genes based on this property. The second class of methods which is more general, often making use of a variety of data sources; these methods start with some existing knowledge of `training' genes already known to be related to the disease under study, and directly or indirectly rank the remaining genes based on their similarity to these training genes. There are also some methods that rank or prioritize genes based on their overall likelihood of being involved in some disease in general.

Those kinds of methods that purpose to improve an initial ranking obtained from expression data by augmenting it with a network structure derived from other data sources can be related to the methods of second class. For example, the GeneRank algorithm of Morrison et al. \cite{morison}, is an intuitive change of PageRank algorithm used by the Google search engine that preserves many of its mathematical properties. It combines gene expression information with a network structure derived from gene annotations (gene ontologies (GO)) or expression profile correlations. In the resulting gene ranking algorithm the ranking of genes can be obtained by solving a large linear system of equations. Wu et al. \cite{wu,wup} showed that this is equivalent to a symmetric positive definite linear system of equations and  analyzed its properties. Then they used the conjugate gradient (CG) algorithm in conjunction with a diagonal scaling to solve the system. Their numerical experiments show that it is the most effective among the tested methods. In this paper, we propose another preconditioner for the GeneRank problem and study its properties.

Throughout this paper, we use the following notations, definitions and results. A matrix $A$ is called nonnegative (positive) and is denoted by $A\geq 0$ ($A>0$) if each entry of $A$ is nonnegative (positive). Similarly, for $n$-dimensional vectors, by identifying them with $n \times 1$ matrices, we can also define $x \geq 0$ and $x>0$. For a square matrix $A$, an eigenvalue of $A$ is denoted by $\lambda(A)$ and the smallest and largest eigenvalues of $A$ are given by $\lambda_{\min}(A)$ and  $\lambda_{\max}(A)$, respectively. For any nonsingular symmetric matrix $A$ we have $Cond(A)=\|A\|_2 \|A^{-1}\|_2= {\lambda_{\max}(A)}/{\lambda_{\min}(A)}$ (see \cite{axelson}).

\begin{definition}\label{Zmatrixdef}
A matrix $A=(a_{ij})\in \Bbb{R}^{n\times n}$ is called a $Z$-matrix if for any $i\neq j$, $a_{ij}\leq 0$. An $M$-matrix is a $Z$-matrix with nonnegative inverse.
\end{definition}

\begin{theorem}\label{Mmatrixthm}
\cite{axelson}
A $Z$-matrix $A=(a_{ij})$ is an $M$-matrix if and only if there exists a positive vector $x$, such that $Ax$ is positive.
\end{theorem}

\begin{theorem}\label{SPDthm}\cite{stoer}
A Hermitian matrix A is positive definite if and only if all eigenvalues of A are positive.
\end{theorem}

The rest of the  paper is organized as follows. In section \ref{SEC2} we introduce the GeneRank problem in details. Section \ref{SEC3} is devoted for presenting our preconditioner. In section \ref{SEC4} we present some numerical experiments to show the effectiveness our preconditioner. Finally, concluding remarks are given in section \ref{SEC5}.

\section{The GeneRank problem formulation}\label{SEC2}

Let the set $G=\{ g_1,...,g_n\} $ be $n$ genes in a microarray. Similar to the idea of PageRank, if a gene is connected with many highly ranked genes, it should be highly ranked as well, even if it may have a low rank from the experimental data. In GO, if two genes share at least one annotation with other genes, they are defined to be connected. From this idea, we can build a gene network whose adjacent matrix is $W$, with entries
\begin{eqnarray}
\label{p11}
W_{ij}=
\left\{
  \begin{array}{ll}
    1, & ~~~g_i~\textrm{and}~g_j~(i \neq j )~\textrm{have}~\textrm{the}~\textrm{same}~\textrm{annotation}~\textrm{in}~\textrm{GO},\\
    0, &~~~ \textrm{otherwise}.
  \end{array}
\right.
\end{eqnarray}
In contrast to PageRank, the connections are not directed. Thus, instead of a nonsymmetric hyperlink matrix, GeneRank considers the symmetric adjacency matrix W of the gene network, i.e., $W^T = W$. Let
\[
deg_i=\sum_{j=1}^n w_{ij}
\]
Note that since a gene might not be connected to any of the other genes, $W$ may have zero rows. Now, suppose that the diagonal matrix $D$ is defined by $D=\textrm{diag} (d_1, \ldots , d_n)$, where
\begin{eqnarray}
\label{wdeg}
d_i=
\left\{
\begin{array}{ll}
deg_i, &~~~ deg_i>0,\\
1,  &~~~\textrm{otherwise}.
\end{array}
\right.
\end{eqnarray}
Then the GeneRank problem can be written as the following large scale nonsymmetric linear system (Morrison et al. \cite{morison})
\begin{eqnarray}
\label{morison}
(I-\alpha W D^{-1})x=(1-\alpha) ex,
\end{eqnarray}
where $I$ denotes the $n\times n$ identity matrix. Also $\alpha$ is the damping factor with $0 < \alpha < 1$, and $ex = [ex_1, ex_2,\ldots , ex_n]^T$ , with $ex_i\geq 0$, is the absolute value of the expression change for $g_i$, $i = 1, 2, \ldots , n$. The solution vector $x$ is called GeneRank vector, and its entries provide information about the significance of a gene.  Morrison et al suggested using $\alpha =0.5$. However, the optimal choice of $\alpha$ is still an interesting topic and deserves further study.

Note that $W$ is a sparse matrix in general. Morrison et al. \cite{morison}, used a direct method  for the computation of GeneRank vector that is extremely inefficient when the problem is very large or $\alpha$ is very close to 1. Yue et al \cite{yue}, reformulated the GeneRank model as a linear combination of three parts, and presented an explicit formulation for the GeneRank vector. In Wu et al \cite{wu}, the GeneRank problem was rewritten as a large scale eigenvalue problem, and it was solved by some Arnoldi-type algorithms. In \cite{wup}, Wu and coworkers proved that the matrix $D-\alpha W$ is symmetric positive definite matrix and showed that the nonsymmetric linear system (\ref{morison}) can be rewritten as the following symmetric positive definite (SPD) linear system
\begin{eqnarray}
\label{spd}
(D-\alpha W)\hat{x}=(1-\alpha) ex,
\end{eqnarray}
with $\hat{x} = D^{-1}x$. Note that equation (\ref{spd}) is equivalent to  equation (\ref{morison}). With this modification, methods that are suitable for symmetric systems, can be used for the GeneRank problem. They implemented the Jacobi preconditioner (symmetric diagonal scaling) on $D -\alpha W$. In that case, the preconditioned system is given by
\begin{eqnarray}
\label{pcg}
(I-\alpha D^{-\frac{1}{2}} W D^{-\frac{1}{2}})\bar{x}=(1-\alpha) D^{-\frac{1}{2}} ex,
\end{eqnarray}
with $\bar{x}= D^{\frac{1}{2}} \hat{x} =  D^{\frac{1}{2}} (D^{-1}x) =  D^{-\frac{1}{2}} x$.\\
\hspace*{.5cm}
Wu and coworkers \cite{wup} also showed that the eigenvalues of the preconditioned matrix are bounded as follows:
\begin{eqnarray}
\label{bound1}
\lambda_{\max} (I-\alpha D^{-\frac{1}{2}} W D^{-\frac{1}{2}})\leq 1+\alpha ,\\
\label{bound2}
\lambda_{\min} (I-\alpha D^{-\frac{1}{2}} W D^{-\frac{1}{2}})\geq 1-\alpha.
\end{eqnarray}
These bounds are independent of the size of the matrix, and they only depend on the value of a parameter used in the GeneRank model.

Recently, Benzi and Kuhlemann in \cite{benzi} showed that coefficient matrices (\ref{morison}) and (\ref{spd}) have a nice property that we introduce it here:
\begin{theorem}\label{GeneRankMthm}
Both of the matrices $D-\alpha W$ and $I+\alpha D^{-\frac{1}{2}} W D^{-\frac{1}{2}}$ are M-matrix for $0<\alpha < 1$.
\end{theorem}

They then implemented the classical Chebyshev iteration for the GeneRank problem. It is a polynomial scheme to accelerate the convergence of the stationary iterative method
\[
x^{(k+1)}=x^{(k)}+M^{-1}(b-Ax^{k}),\quad k=0,1,\ldots,
\]
to solve a linear system of equation $Ax=b$, in which $M$ is a nonsingular matrix. If the matrix  $I-M^{-1}A$ is similar to a symmetric matrix with eigenvalues lying in an interval $[l_{\min},l_{\max}]$, then the Chebyshev iteration to solve $Ax=b$ is given by (for more details, see  \cite{benzi,Golub})
\begin{eqnarray*}
y^{k+1} &=&\frac{\omega_{k+1}}{2-(l_{\min}+l_{\max})} \left\{2M^{-1}(b-Ay^{(k)})+[2-(l_{\min}+l_{\max})](y^{(k)}-y^{(k-1)}) \right\}\\
        & & +y^{(k-1)},\quad k=1,2,\ldots,
\end{eqnarray*}
where $y^{(0)}=x^{(0)}$, $y^{(1)}=x^{(1)}$ and
\[
\omega_{k+1}=\frac{1}{1-\frac{w_k^2}{4\omega^2}},\quad w_2=\frac{2\omega^2}{2\omega^2-1},\quad \omega_1=1,\quad \omega=\frac{2-(l_{\min}+l_{\max})}{l_{\max}-l_{\min}}.
\]
For the GeneRank problem, it is considered $l_{\min}=1-\alpha$, $l_{\max}=1+\alpha$, $A=D-\alpha W$, $b=ex$ and $M=D=\textrm{diag}(A)$.

Numerical results presented in \cite{benzi} show that the number of iterations of the method is usually more than those of the  CG method with the diagonal scaling. However, the cost per iteration is much lower and leads to faster convergence in terms of CPU time.

\section{Main results} \label{SEC3}

Wu et al. \cite{wu,wup} showed that coefficient matrix of linear system (\ref{spd}) is symmetric positive definite and therefore we can use the conjugate gradient method (\cite{cg}) to solve the system. As we know that the convergence rate of the CG method depends on the condition number of the matrix in question, or more generally the distribution of eigenvalues. If the eigenvalue distribution of the preconditioned system is better than that of the original one, the convergence will be accelerated dramatically. For this reason Wu and coworkers applied Jacobi preconditioner on the system (\ref{spd}) to accelerate in convergence of the method. In this case, the linear system (\ref{pcg}) were resulted. Considering (\ref{bound1}) and (\ref{bound2}), we can conclude that increasing $\alpha$ from 0 to 1 can cause an increase in the ratio of ${\lambda_{\max}}/{\lambda_{\min}}$ and therefore the coefficient matrix would be increasingly ill-conditioned and the rate of convergence of CG expected to decrease as $\alpha$ increases.

From now on, for the sake of the  simplicity, let  $J_{\alpha}=\alpha D^{-\frac{1}{2}} W D^{-\frac{1}{2}}$ and $S_{\alpha}=I-J_{\alpha}$. It is noted that the matrix $S_\alpha$ is the coefficient matrix of the system (\ref{pcg}). We now propose the preconditioner $M_{\alpha}=I+J_{\alpha}$ for the system (\ref{pcg}). In this case, the preconditioned system takes the form
\begin{equation}\label{preconme}
    M_{\alpha} S_{\alpha} \bar{x}=M_{\alpha}b_{\alpha},
\end{equation}
where $b_{\alpha}=(1-\alpha) D^{-\frac{1}{2}} ex$. In the sequel, we investigate the properties of the proposed preconditioner.

\begin{theorem}\label{TSPDM}
For every $0<\alpha < 1$, the matrix $M_{\alpha} S_{\alpha}$ is symmetric positive definite M-matrix.
\end{theorem}
\begin{proof}
Clearly, we have
\[
T_{\alpha}:=M_{\alpha} S_{\alpha}=I-J_{\alpha}^2=I-(I-S_{\alpha})^2.
\]
Therefore,
\[
\lambda(T_{\alpha})=1-(1-\lambda(S_{\alpha}))^2.
\]
According to the relations  (\ref{bound1}) and (\ref{bound2}) we have $0<\lambda(S_{\alpha}) <2$. Hence, by the latter equation we deduce that $0<\lambda(T_{\alpha}) <1$. Now since $T_{\alpha}$ is a symmetric matrix with positive eigenvalues, it follows that $T_{\alpha}$ is symmetric positive definite. Obviously, $T_{\alpha}$ is a Z-matrix. From Theorem \ref{GeneRankMthm},  $S_\alpha$ is an $M$-matrix. Therefore, by Theorem \ref{Mmatrixthm}, there exists a positive vector $x$ such that $S_{\alpha}x>0$. Hence, $T_{\alpha}x=(I+J_{\alpha})S_{\alpha}x>0$. This shows that the matrix $T_{\alpha}$ is an M-matrix.
\end{proof}

\begin{lemma}\label{SmalleigSLem}
If $W\neq 0$, then $\lambda_{\min}(S_{\alpha})=1-\alpha$.
\end{lemma}
\begin{proof}
Let $x=(x_1,\ldots,x_n)^T$, such that
\begin{eqnarray*}
x_i=
\left\{
\begin{array}{ll}
1, &~~~ deg_i>0,\\
0, &~~~ \textrm{otherwise},
\end{array}
\right.
\end{eqnarray*}
for $i=1,2,\ldots,n$. Obviously, we have $Wx=Dx$. Now, assuming $y=D^{\frac{1}{2}}x\neq 0$, we get
\begin{eqnarray*}
(I- \alpha D^{-\frac{1}{2}}WD^{-\frac{1}{2}})y=(1-\alpha)y,
\end{eqnarray*}
which is equivalent to $S_{\alpha}y=(1-\alpha)y$.  Now by Eq. (\ref{bound2}), we see that  $1-\alpha$ is the smallest eigenvalue of $S_{\alpha}$.
\end{proof}

\begin{remark}
In the case that $W=0$, we have $S_\alpha=I$ and there is nothing  to investigate.
\end{remark}

\begin{theorem}
\label{p2}
The spectral condition number of $T_{\alpha}$ is not greater than that of the matrix $S_{\alpha}$, i.e.,
\[
Cond_{2}(T_{\alpha}) \leq Cond_{2}(S_{\alpha}).
\]
\end{theorem}
\begin{proof}
Since both of the matrices $S_{\alpha}$ and $T_{\alpha}$ are symmetric,  it is enough to prove that
\[
\frac{ \lambda_{\max}{(T_{\alpha})}  }{ \lambda_{\min}{(T_{\alpha}} )}\leq \frac{ \lambda_{\max}{(S_{\alpha})}  }{ \lambda_{\min}{(S_{\alpha}} )}.
\]
For every eigenvalue $\lambda(S_{\alpha})$ of $S_{\alpha}$, we have  $\lambda_{\min}(S_{\alpha}) \leq \lambda(S_{\alpha}) \leq \lambda_{\max}(S_{\alpha})$.
Then,
\begin{eqnarray*}
\label{p21}
1-\lambda_{\max}(S_{\alpha})\leq 1-\lambda(S_{\alpha})\leq 1-\lambda_{\min}(S_{\alpha}).
\end{eqnarray*}
From Lemma \ref{SmalleigSLem}, we have $\lambda_{\min}(S_{\alpha}) \leq 1$. We now consider two cases, $\lambda_{\max}(S_{\alpha}) \geq 1$ and  $\lambda_{\max}(S_{\alpha}) < 1$. If $\lambda_{\max}(S_{\alpha}) \geq 1$, then by Eq. (\ref{bound1}) and Lemma \ref{SmalleigSLem} we have
$\lambda_{\max}(S_{\alpha})-1 \leq 1-\lambda_{\min}(S_{\alpha})$.
Therefore $1-(1-\lambda_{\max}(S_{\alpha}))^2 \geq 1-(1-\lambda_{\min}(S_{\alpha}))^2$ and since the maximum value of $1-(1-\lambda(S_{\alpha}))^2$ is equal to 1, we get
\begin{eqnarray*}
Cond(T_{\alpha})\leq \frac{1}{1-(1-\lambda_{\min}(S_{\alpha}))^2}\leq \frac{\lambda_{\max}(S_{\alpha})}{\lambda_{\min}(S_{\alpha})}=Cond(S_{\alpha})
\end{eqnarray*}
Now, suppose that $\lambda_{\max}(S_{\alpha})\leq 1$. In this case, it is easy to see that
\begin{eqnarray*}
Cond(T_{\alpha})=\frac{1-(1-\lambda_{\max}(S_{\alpha}))^2}{1-(1-\lambda_{\min}(S_{\alpha}))^2}\leq \frac{\lambda_{\max}(S_{\alpha})}{\lambda_{\min}(S_{\alpha})}=Cond(S_{\alpha})
\end{eqnarray*}
Therefore, the proof of theorem is completed.
\end{proof}

This theorem shows that the eigenvalues of the matrix $T_{\alpha}$ are clustered at least as the matrix $S_{\alpha}$.
As we shall see, for the presented numerical experiments the eigenvalues of $T_{\alpha}$ are more clustered than those of the matrix $S_{\alpha}$.

\begin{remark}
From Lemma \ref{SmalleigSLem}, if $W\neq 0$, then $\lambda_{\min}(S_{\alpha})=1-\alpha$. Therefore, according to the relation $\lambda(T_{\alpha})=1-(1-\lambda(S_{\alpha}))^2$ we have $\lambda_{\min}(T_{\alpha})=1-\alpha^2$.
Having in mind that $0<\lambda(T_{\alpha}) <1$, we deduce that $1-\alpha^2\leq \lambda(T_{\alpha}) <1$. Note that $1-\alpha\leq \lambda(S_{\alpha}) \leq 1+\alpha$.
\end{remark}

\section{Numerical experiments}\label{SEC4}

As we mentioned,  Wu et al. in \cite{wup} successfully employed the Jacobi preconditioner together with the CG algorithm for the solution of the linear system. They have deduced that the  CG method in conjunction with the Jacobi preconditioner was faster among the tested methods for each of the presented examples. In this section we compared the numerical results of our preconditioner with those of the Jacobi preconditioner.

All the numerical experiments presented in this section were computed in double precision and the algorithms are implemented in MATLAB 7.12.0 and tested on a 64-bit 1.73 GHz intel Q740 core i7 processor and 4GB RAM running windows 7. We use the stopping criteria based on the 1-norm of the residual. That is, we stop iterating as soon as $\| r \|_1 < tol$, where $tol$ is a given tolerance. The initial guess is the null  vector. We use two different choices for $ex$: $ex=(\frac{1}{n})e$, where $e$ is the vector of all ones, and $ex = p$, where $p$ is a randomly chosen probability vector --that is, a random vector with entries in $(0,1)$. For each adjacency matrix, we use four different values of $\alpha$ to form the corresponding GeneRank matrices $D - \alpha W$: $\alpha= 0.5, 0.75, 0.80, 0.99$. The obtained numerical results are presented in some tables. In all the tables, ``CG", ``PCG", ``Chebyshev" and ``CG-$M_{\alpha}$" are, respectively, denoted for the CG method implemented for the system (\ref{spd}), CG method applied to the system (\ref{pcg}), Chebyshev iteration and the CG algorithm to the system (\ref{spd}) in conjunction with the preconditioner $M_{\alpha}$.

\begin{example}\rm
\label{w}
In this example there are three adjacency matrices,  \textit{w\_All} which is of size $4, 047 \times 4, 047$, with $339, 596$ nonzero entries, \textit{w\_Up} which is of size $2, 392 \times 2, 392$, with $128, 468$ nonzero entries and \textit{w\_Down} which is of size $1, 625 \times 1, 625$ with $67, 252$ nonzero entries and three expression change vectors \textit{expr\_data}, \textit{expr\_dataUp} and \textit{expr\_dataDown}. These matrices were constructed using the all three sections of the GO, where a link is presented between two genes if they share a GO annotation. Only genes which are up-regulated are included in \textit{w\_Up} and only down-regulated in \textit{w\_Down} \cite{morison}. The data files are available from \cite{wall}. The results for the these matrices are given in Tables 4.1, 4.2 and 4.3.  In these tables (hereafter for other tables) the number of iterations for the convergence  together with the CPU time (in parenthesis) in seconds are given. Here we mention that, in this example, the tolerance $tol$ is taken to be $10^{-14}$. As the numerical experiments show our preconditioner is more effective than the Jacobi preconditioner. It seems that the number of iterations of our method is about half of those of the CG method in conjunction with the Jacobi preconditioner. In addition, as seen our preconditioner outperforms in terms of both number of iterations and CPU times.

For more investigation, in Figures 4.1, 4.2 and 4.3  we depict the eigenvalues distribution of  $S_\alpha$ and $T_\alpha$ for test matrices  $w\_Down$ when $\alpha = 0.5$, $w\_Up$ when $\alpha=0.75$ and $w\_All$ when $\alpha=0.9$, respectively. As seen the eigenvalues of the $T_\alpha$ are more clustered than those of $S_{\alpha}$ for all the three  test matrices.

\begin{figure}
{\includegraphics[height=9cm,width=13cm]{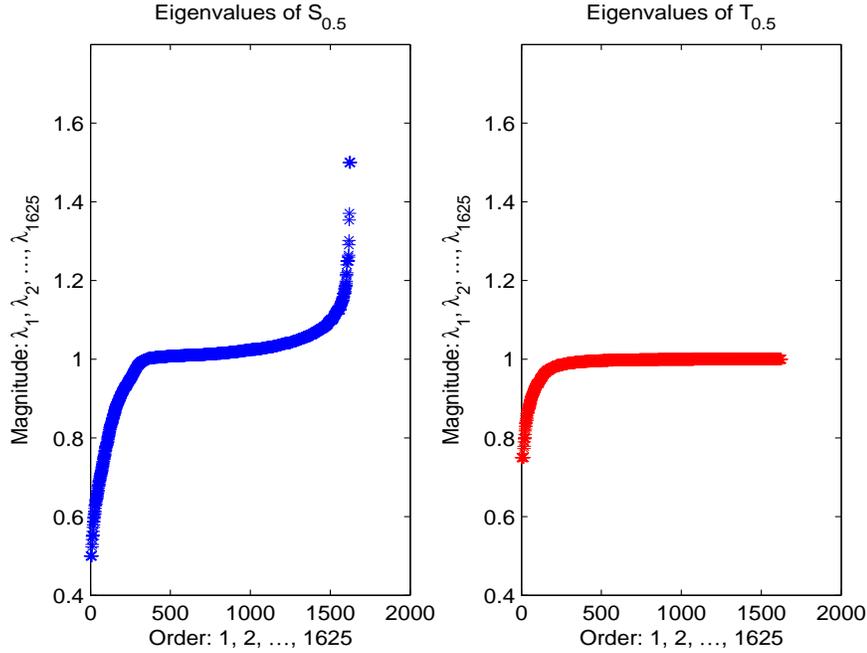}}
{\caption{Eigenvalue distribution of $S_{0.5}$ and $T_{0.5}$ for the matrix $w\_Down$.}}\label{Fig1}\vspace{1cm}
\end{figure}
\begin{figure}
\centerline{\includegraphics[height=9cm,width=13cm]{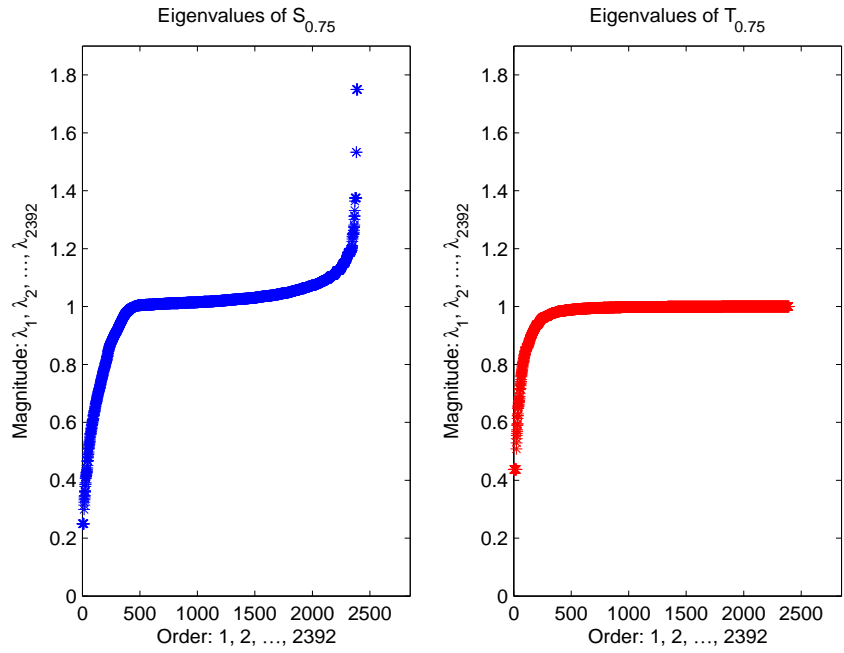}}
{\caption{Eigenvalue distribution of $S_{0.75}$ and $T_{0.75}$ for the matrix $w\_Up$.}}\label{Fig2}\vspace{1cm}
\end{figure}
\begin{figure}
\centerline{\includegraphics[height=9cm,width=13cm]{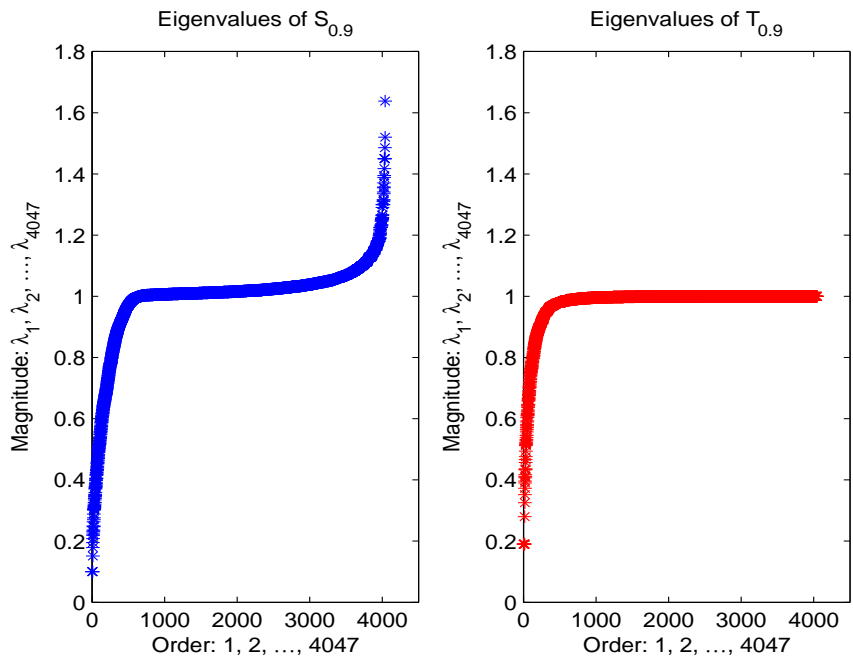}}
{\caption{Eigenvalue distribution of $S_{0.9}$ and $T_{0.9}$ for the matrix $w\_All$.}}\label{Fig3}
\end{figure}
\end{example}

\begin{table}\label{Table11}
\caption{Results for the $w\_All$ matrix. Here $ex=extr\_data$.}
\vspace{0.cm}
\begin{tabular}{lrrrrrr}\hline
$\alpha$ ~~~~~~ &~~~~~~ 0.50 ~~~~~&~~~~~~ 0.75 ~~~~~&~~~~~~ 0.80 ~~~~~&~~~~~~ 0.99~~~~~ \\\hline
CG                &  $381(0.700)$  & $441(0.772)$ & $456(0.785)$  & $603(1.012)$ \\
PCG               &  $26(0.051)$   & $39(0.082)$  & $42(0.081)$   & $69(0.133)$ \\
Chebyshev         & $23(0.030)$  & $36(0.048)$  & $41(0.057)$  & $177(0.195)$\\
CG-$M_{\alpha}$   & $11(0.023)$  & $16(0.037)$  & $18(0.044)$  & $30(0.064)$\\\hline
\end{tabular}
\end{table}

\begin{table}\label{Table21}
\caption{Results for the $w\_Up$ matrix. Here $ex=extr\_dataUp$.}
\vspace{-0.cm}
\begin{tabular}{lrrrrrr}\hline
$\alpha$ ~~~~~~ &~~~~~~ 0.50 ~~~~~&~~~~~~ 0.75 ~~~~~&~~~~~~ 0.80 ~~~~~&~~~~~~ 0.99~~~~~~ \\\hline
CG                &  $309(0.142)$  & $360(0.162)$ & $377(0.173)$   & $488(0.221)$ \\
PCG               &  $26(0.017)$ & $39(0.026)$ & $42(0.022)$   & $70(0.033)$ \\
Chebyshev         & $22(0.008)$  & $36(0.013)$  & $41(0.014)$  & $177(0.058)$\\
CG-$M_{\alpha}$   & $11(0.006)$  & $16(0.011)$  & $18(0.012)$  & $31(0.017)$\\\hline
\end{tabular}
\end{table}

\begin{table}\label{Table31}
\caption{Results for the $w\_Down$ matrix. Here $ex=extr\_dataDown$. }
\vspace{-0.cm}
\begin{tabular}{lrrrrrr}\hline
$\alpha$ ~~~~~~ &~~~~~~ 0.50 ~~~~~~&~~~~~~ 0.75 ~~~~~~&~~~~~~ 0.80 ~~~~~~&~~~~~~ 0.99~~~~~~ \\\hline
CG                &  $267(0.087)$  & $310(0.091)$ & $322(0.094)$  & $427(0.131)$ \\
PCG               &  $27(0.010)$   & $40(0.013)$  & $44(0.014)$   & $76(0.024)$ \\
Chebyshev         & $22(0.007)$  & $35(0.008)$  & $40(0.009)$  & $173(0.035)$\\
CG-$M_{\alpha}$   & $11(0.006)$  & $17(0.006)$  & $18(0.006)$  & $32(0.013)$\\\hline
\end{tabular} \end{table}

\begin{example} \rm
In this example, we use two different types of test data for our experiments. The first matrix is a SNPa adjacency matrix (single-nucleotide polymorphism matrix). This matrix has $n = 152,520$ rows and columns, and is very sparse with only 639,248 nonzero entries. The second type is RENGA adjacency matrix (range-dependent random graph model). In our experiments we set $\lambda = 0.9$ and $\beta = 1$, the default values in RENGA. Both of these types of matrices are tested in \cite{wup}. The results for the SNPa matrix are given in Tables 4.4 and 4.5, and the results for the RENGA matrix with $n = 100, 000$ and $n=500, 000$ are given in Tables 4.6 and 4.7. In this example the tolerance $tol$ is assumed to be $10^{-10}$. As the numerical experiments show the proposed preconditioner is more effective tahn the diagonal scaling and Chebyshev iteration method.
\end{example}

\begin{table}\label{Table12}
\caption{Results for the SNPa matrix. Here $ex = (\frac{1}{n})e$, where $e$ is the vector of all ones. }
\vspace{-0.cm}
\begin{tabular}{lrrrrrr}\hline
$\alpha$ ~~~~~~   &~~~~~~~~ 0.50 ~~~~&~~~~~~~~ 0.75 ~~~~&~~~~~~~~ 0.80 ~~~~&~~~~~~~~ 0.99~~~~ \\\hline
CG                & $86(2.420)$  & $116(3.321)$ & $128(3.663)$ & $469(13.321)$\\
PCG               & $17(0.514)$  & $27(0.818)$  & $30(0.922)$  & $91(2.738)$\\
Chebyshev         & $17(0.208)$  & $28(0.359)$  & $31(0.384)$  & $130(1.402)$\\
CG-$M_{\alpha}$   & $9(0.164)$  & $13(0.213)$  & $15(0.236)$  & $44(0.704)$\\\hline
\end{tabular} \end{table}
\begin{table}\label{Table22}
\caption{Results for the SNPa matrix. Here $ex = p$, where $p$ is a random   probability vector.  }
\vspace{-0.cm}
\begin{tabular}{lrrrrrr}\hline
$\alpha$ ~~~~~~ &~~~~~~~~ 0.50 ~~~~&~~~~~~ 0.75 ~~~~&~~~~~~ 0.80 ~~~~&~~~~~~ 0.99~~~~ \\\hline
CG                & $127(3.649)$  & $174(5.011)$ & $193(5.607)$ & $721(20.411)$\\
PCG               & $26(0.789)$  & $41(1.242)$  & $46(1.375)$  & $139(4.203)$\\
Chebyshev         & $17(0.211)$  & $27(0.354)$  & $30(0.352)$  & $125(1.368)$\\
CG-$M_{\alpha}$   & $8(0.124)$  & $13(0.209)$  & $14(0.232)$  & $43(0.713)$\\\hline
\end{tabular}
\end{table}
\begin{table}\label{Table32}
\caption{Results for the RENGA matrices.  Here $ex = (\frac{1}{n})e$, where $e$ is the vector of all ones.}
\vspace{-0.cm}
\begin{tabular}{lrrrrrr}\hline
~~~~~~~~& ~~~~~~~~& $n=100, 000$ & ~~~~~~\\\hline
$\alpha$ ~~~~~~~~ &~~~~~~ 0.50 ~~~~&~~~~~~ 0.75 ~~~~&~~~~~~ 0.80 ~~~~&~~~~~~ 0.99~~~~ \\\hline
CG                &  $43(0.789)$ & $47(0.840)$ & $48(0.863)$   & $129(2.305)$ \\
PCG               &  $14(0.262)$ & $22(0.431)$ & $24(0.459)$   & $95(1.824)$ \\
Chebyshev         & $17(0.183)$  & $27(0.245)$  & $31(0.259)$  & $127(0.905)$\\
CG-$M_{\alpha}$   &  $8(0.092)$ & $12(0.141)$ & $14(0.168)$    & $53(0.673)$ \\\hline
~~~~~~~~& ~~~~~~~~& $n=500, 000$& ~~~~~~~\\\hline
CG                &  $23(2.612)$ & $27(3.006)$ & $30(3.398)$   & $106(12.019)$ \\
PCG               &  $13(1.578)$ & $20(2.453)$ & $22(2.661)$   & $86(10.497)$ \\
Chebyshev         & $17(0.901)$  & $27(1.371)$  & $30(1.512)$  & $125(6.359)$\\
CG-$M_{\alpha}$   &  $7(0.645)$ & $12(1.100)$ & $14(1.234)$    & $50(4.671)$ \\\hline
\end{tabular}
\end{table}
\begin{table}\label{Table42}
\caption{Results for the RENGA matrices. Here $ex = p$, where $p$ is a \newline random probability vector. }
\vspace{-0.cm}
\begin{tabular}{lrrrrrr}\hline
~~~~~~~~& ~~~~~~~~& $n=100, 000$ & ~~~~~~\\\hline
$\alpha$ ~~~~~~~~ &~~~~~~ 0.50 ~~~~&~~~~~~ 0.75 ~~~~&~~~~~~ 0.80 ~~~~&~~~~~~ 0.99\\\hline
CG                &  $68(1.254)$  & $73(1.299)$ & $75(1.337)$   & $222(4.021)$ \\
PCG               &  $22(0.418)$ & $34(0.667)$ & $39(0.732)$   & $165(3.286)$ \\
Chebyshev         & $17(0.167)$  & $26(0.255)$  & $30(0.269)$  & $122(0.883)$\\
CG-$M_{\alpha}$   &  $8(0.093)$ & $12(0.145)$ & $14(0.170)$    & $53(0.668)$ \\\hline
~~~~~~~~& ~~~~~~~~& $n=500, 000$& ~~~~~~~\\\hline
CG                &  $36(4.101)$ & $45(5.109)$ & $50(5.637)$   & $200(22.676)$ \\
PCG               &  $21(2.509)$ & $34(4.168)$ & $38(4.585)$   & $163(19.841)$ \\
Chebyshev         & $16(0.838)$  & $26(1.360)$  & $29(1.455)$  & $120(6.062)$\\
CG-$M_{\alpha}$   &  $8(0.733)$ & $12(1.094)$ & $13(1.174)$    & $51(4.761)$ \\\hline
\end{tabular}
\end{table}

\section{Conclusion}\label{SEC5}

In this paper,  a preconditioner has been presented for the GeneRank problem.  Then some of its properties have been given. Finally some numerical experiments have been presented to show the effectiveness of the proposed preconditioner. As seen it does not need any CPU time to set up the preconditioner and the preconditioner is  explicitly in hand.
Our numerical results show that the proposed preconditioner is more effective than the ones presented in the literature.

\section*{Acknowledgements}

 We would like to thank Prof. Yimin Wei from  Fudan University for
providing us the SNPa data and Prof. Michele Benzi from Emory University to provide us the code of Chebyshv acceleration method.


\begin{thebibliography}{1}

\bibitem{axelson} O. Axelsson, Iterative solution methods, Cambridge University Press, Cambridge, 1996

\bibitem{wall} S. Agarwal and S. Sengupta, Ranking genes by relevance to a disease, Proc. LSS Comput. Syst. Bioinform. Conf. 8 (2009) 37-–46.

\bibitem{benzi} M. Benzi and V. Kuhlemann, Chebyshev acceleration of the GeneRank algorithm, Electronic Transactions on Numerical Analysis 40 (2013) 311-320.

\bibitem{Golub}  G.H. Golub and C.F. van Loan, Matrix Computations, 3rd edition, Johns Hopkins University Press,
Baltimore, 1996.

\bibitem{cg} M. R. Hestenes and E. Stiefel, Methods of conjugate gradients for solving linear systems, Journal of Research of the National Bureau of Standards  49 (1952) 409-436.

\bibitem{morison} J. Morrison, R. Breitling, D. Higham and D. Gilbert,  GeneRank: using search engine for the analysis of microarray experiments, BMC Bioinform,  6 (2005) 233-–246.

\bibitem{saad} Y. Saad. Iterative methods for sparse linear systems, 2nd Edition. SIAM, Philadelphia, 2003.

\bibitem{stoer} J. Stoer, R. Bulirsch, Introduction to numerical analysis, New York: Springer-Verlag, 1980.

\bibitem{wu} G. Wu, Y. Zhang, and Y.Wei. Krylov subspace algorithms for computing GeneRank for the analysis of microarray data mining. Journal of Computational Biology 17 (2010) 631–646.

\bibitem{wup} G. Wu, W. Xu, Y. Zhang, and Y. Wei, A preconditioned conjugate gradient algorithm for GeneRank with application to microarray data mining, Data Min. Knowl. Disc. 26 (2013) 27-–56.

\bibitem{yue} B. Yue, H. Liang and F. Bai,  Understanding the GeneRank model, IEEE 1st Int Conf Bioinform Biomed Eng. 6-8 (2007) 248–251.


\end{thebibliography}
\end{document}